\documentclass[11pt,reqno]{amsart}
\usepackage[utf8]{inputenc}
\usepackage{amsmath,amsthm,amssymb,tikz-cd,mathrsfs,enumitem,caption,subcaption,xfrac}
\usepackage[scale=0.75]{geometry}

\title{Uniform boundedness on rational maps with automorphisms}
\author{Minsik Han}
\address{Department of Mathematics, University of Rochester, Rochester, NY 14627, USA}
\email{minsik.han@rochester.edu}

\makeatletter
\@namedef{subjclassname@2020}{%
  \textup{2020} Mathematics Subject Classification}
\makeatother

\newcommand{\bC}{\mathbb{C}}

\newcommand{\bP}{\mathbb{P}}
\newcommand{\bQ}{\mathbb{Q}}
\newcommand{\bZ}{\mathbb{Z}}

\newcommand{\cP}{\mathcal{P}}

\newcommand{\Aut}{\textup{Aut}}
\newcommand{\ord}{\textup{ord}}
\newcommand{\PrePer}{\textup{PrePer}}
\newcommand{\PGL}{\textup{PGL}}

\newtheorem{theorem}{Theorem}

\newtheorem{lemma}[theorem]{Lemma}
\newtheorem{conjecture}[theorem]{Conjecture}

\theoremstyle{remark}
\newtheorem*{remark}{Remark}
\theoremstyle{definition}
\newtheorem*{definition}{Definition}

\begin{document}

\begin{abstract}
    In this paper, we study the dynamical uniform boundedness principle over a family of rational maps with certain nontrivial automorphisms. Specifically, we consider a family of rational maps of an arbitrary degree $d\ge 2$ whose automorphism group contains the cyclic group of order $d$. We prove that a subfamily of this family satisfies the dynamical uniform boundedness principle.
\end{abstract}

\maketitle

\section{Introduction}

In 1950, Northcott \cite{northcott} proved a significant result that for every $N\ge 1$ and every number field $K$, any morphism $\phi:\bP^N(K)\to \bP^N(K)$ of degree $d\ge 2$ has a finite number of preperiodic points.
This result led to the following Dynamical Uniform Boundedness Conjecture, which is one of the most significant conjectures in arithmetic dynamics.

\begin{conjecture}[Morton-Silverman, 1994, \cite{ubc}]
Let $d\ge 2$ and $N,D\ge 1$ be integers.
Then there exists a constant $C=C(d,N,D)$ such that
\[\# \PrePer \left(\phi, \bP^N(K)\right) \le C\]
for any number field $K$ satisfying $[K:\bQ]\le D$ and any finite morphism $\phi:\bP^N(\overline{K}) \to \bP^N(\overline{K})$ of degree $d$ defined over $K$.
\end{conjecture}

There have been various approaches to deal with this conjecture.
For example, Looper \cite{looper} recently proved this conjecture for the family of polynomial maps on $\bP^1$, assuming a generalization of the $abc$-conjecture called the $abcd$-conjecture.
Without this assumption, this conjecture has not even been proved for the family of quadratic unicritical polynomial maps on $\bP^1(\bQ)$.

\begin{conjecture} \label{conj:uniformquad}
Let $\phi_c(z)=z^2+c$ be a quadratic polynomial map defined over $\bQ$.
Then there exists a constant $C$ such that for any $c\in \bQ$,
\[\# \PrePer\left (\phi_c, \bP^1(\bQ)\right) \le C.\]
\end{conjecture}

Regarding this conjecture, there has been work to try to find the upper bound for the length of rational periodic cycles.
While there exist (infinitely many) $c$-values such that $\phi_c$ has a rational periodic cycle of length $1$, $2$, or $3$, there are no $\phi_c$ that have a rational periodic cycle of length $4$ \cite{morton} or $5$ \cite{fps}.
Stoll \cite{stoll} proved that there cannot be a rational periodic cycle of length $6$ as well, assuming the Birch and Swinnerton-Dyer conjecture for a specific abelian variety.
Based on these results, Poonen \cite{poonen} proved that if $1$, $2$, and $3$ are the only possible lengths of rational periodic cycles of $\phi_c$, then Conjecture \ref{conj:uniformquad} is true with $C=9$.

There has also been work on other families of rational maps.
In 2014, Levy-Manes-Thompson \cite{manes} studied a specific family of quadratic rational maps defined over a number field $K$,
\begin{equation} \label{eq:manesmap}
    \psi_{a,b}(z)= b \left(z+\frac{a}{z}\right),\quad a,b\in \bQ^*.
\end{equation}
Note that their parameters have been relabeled to be consistent with this paper.
For $K=\bQ$, they proved the following.

\begin{theorem}[Levy-Manes-Thompson, 2014, \cite{manes}]
\label{theorem:LMT2014}
Let $b\in\bQ^*$. Then any rational periodic cycle of $\psi_{a,1}$ or $\psi_{a,-1}$ has length at most $4$. Furthermore, for any $a \in \bQ^*$, we have \[\#\PrePer\left(\psi_{a,\pm 1}, \bP^1(\bQ) \right) \le 6.\]
\end{theorem}

Manes \cite{manes2} also worked on these maps with arbitrary $b\in \bQ^*$ and proved the following uniform boundedness result.

\begin{theorem}[Manes, \cite{manes2}]
    For any $a,b \in \bQ^*$, if there is no rational periodic cycle of length greater than $4$, then \[\#\PrePer\left(\psi_{a,b}, \bP^1(\bQ) \right) \le 12.\]
\end{theorem}

The rational map $\psi_{a,b}$ has a nontrivial automorphism group over $\bC$. In particular, for any $a,b\in \bQ^*$, \[\bZ/2\bZ \hookrightarrow \Aut_{\bC}(\psi_{a,b}).\] It can be naturally extended to the family of rational maps of degree $d$, whose automorphism group contains a copy of $\bZ/d\bZ$. Miasnikov-Stout-Williams \cite{auto} proved that this family of rational maps is also a single-parameter family in the moduli space of rational maps of degree $d$.

Regarding this, we extend Theorem \ref{theorem:LMT2014} to rational maps with higher degrees and more choices of parameters. Below are our main results.

\begin{theorem} \label{thm:main1}
Let $d \ge 2$ and $p$ be a prime. Also, let \[\cP_0^+ := \bigl\{ \sigma p^e : \sigma \in \{\pm 1\},\ e \in \bZ_0^+ \bigr\}.\] Then any rational periodic cycle of $\psi_{d,a,b}:\bP^1 \to \bP^1$ defined as \[ \psi_{d,a,b} (z) = \frac{z}{az^d+b},\quad a\in \bQ^*,\ b\in \cP_0^+\] has length at most $2$.
\end{theorem}

\begin{theorem} \label{thm:main2}
    Under the same assumptions on $\psi_{d,a,b}$ as in Theorem \ref{thm:main1}, \[\PrePer \Bigl( \psi_{d,a,b} , \bP^1(\bQ) \Bigr) \le 6.\]
\end{theorem}

To prove these results, we use the dynatomic polynomials associated to $\psi_{d,a,b}$. Due to the structure of its nontrivial automorphisms, we can reduce the dynatomic polynomial to a simpler form by a change of variable (Lemma \ref{lemma:phibntilde}). The resulting polynomial has integer coefficients, where in particular the leading coefficient and constant are prime powers up to sign. This leads us to prove Theorem \ref{thm:main1}. Finally, we investigate polynomials associated to possible preperiodic structures using Newton polygons to derive Theorem \ref{thm:main2}.

There is no known results for the $b$-values which are not prime powers. Also, due to some techniques used in the proofs, such as the rational root theorem, these theorems cannot be extended to other number fields. For example, for any given $d\ge 2$, there exist $a,b \in \bQ(\zeta_d)$ such that \[\#\PrePer \left(\psi_{d,a,b},\bP^1(\bQ(\zeta_d))\right) \ge 2d+2\] where $\zeta_d$ is a $d$-th root of unity. We hope to use other tools, such as heights of the coefficients of dynatomic polynomials, to find a more general bound.

\section{Main results}

Let $d \ge 2$ be an integer, and consider a rational map of the form
\begin{equation} \label{eq:psidehom}
\psi_{d,a,b}(z) = \frac{z}{az^d+b}
\end{equation}
with $a,b\in \bQ^*$.
Note its homogenized form
\begin{equation} \label{eq:psimap}
\psi_{d,a,b} \bigl([X,Y]\bigr) = [XY^{d-1},aX^d+bY^d],\quad a,b \in \bQ^*.
\end{equation}
This map has a $\bC$-automorphism group containing a copy of $\bZ/d\bZ$, generated by \[z\mapsto \zeta_d z\] where $\zeta_d$ is a primitive $d$-th root of unity.
Also, $\psi_{d,a,b}$ is $\PGL_2(\bC)$-conjugate to $\psi_{d,1,b}$ by the map
\[z \mapsto \frac{z}{\sqrt[d]{a}},\]
so $\psi_{d,a,b}$ is a twist of $\psi_{d,1,b}$.

For each $a,b\in \bQ^*$ we consider the homogeneous dynatomic polynomial \[\Phi_{d,a,b,n}^* := \Phi_{\psi_{d,a,b},n}^*(X,Y) = \prod_{m\mid n} \bigl(YF_m(X,Y)-XG_m(X,Y)\bigr)^{\mu(n/m)} \in (\bZ[a,b])[X,Y],\] where $F_n,G_n$ are defined recurrently by
\begin{equation*} \begin{aligned}
    F_0 & = X, & G_0 &= Y, \\
    F_n & = F_{n-1}G_{n-1}^{d-1}, & \hspace*{2em} G_n & = aF_{n-1}^d + bG_{n-1}^d.
\end{aligned} \end{equation*}

To simplify the situation, we consider a different dynamical system.
Define $\widetilde F_n ,\widetilde G_n \in (\bZ[b])(x,y)$ recurrently by
\begin{equation} \label{eq:tildefgrecurrence}
\begin{aligned}
    \widetilde F_0 & = y, &&& \widetilde G_0 & = x, \\
    \widetilde F_n & = \widetilde F_{n-1} \widetilde G_{n-1}^{d-1}, &&&\hspace*{2em} \widetilde G_n & = (x-by) \frac{\widetilde F_{n-1}^d}{y} + b \widetilde G_{n-1}^d.
\end{aligned}
\end{equation}
Then it follows from induction that $y \mid \widetilde F_n$ for all $n$, so $\widetilde F_n, \widetilde G_n \in (\bZ[b])[x,y]$.
Moreover, both $\widetilde F_n, \widetilde G_n$ are homogeneous in $x,y$ with degree $d^n$.

We construct a dynatomic-style polynomial that is related to the dynatomic polynomial corresponding to $\psi_{d,a,b}$ defined in \eqref{eq:psimap}.

\begin{lemma}
\label{lemma:phibntilde}
For $n \ge 1$, define
\begin{equation} \label{eq:newdyn}
    \widetilde \Phi_{d,b,n}^* (x,y) = \prod_{m\mid n} (\widetilde F_{m-1} - \widetilde G_{m-1})^{\mu(n/m)}.
\end{equation}
Then $\widetilde \Phi_{d,b,n}^* \in (\bZ[b])[x,y]$ satisfies
\begin{equation} \label{eq:dynrelation}
    \widetilde \Phi_{d,b,n}^*(aX^d+bY^d,Y^d) = \Phi_{d,a,b,n}^*(X,Y)
\end{equation}
for all $n>1$.
\end{lemma}

\begin{proof}
We first prove that by induction that for all $n\ge 1$ we have
\begin{equation} \label{eq:tildefgrel}
\begin{aligned}
    \widetilde F_{n-1} (aX^d+bY^d,Y^d) &= \frac{YF_n(X,Y)}{X},\\
    \widetilde G_{n-1} (aX^d+bY^d,Y^d) &= G_n(X,Y).\\
    \end{aligned}
\end{equation}

\begin{itemize}

\item \framebox{$n=1$}\enspace Since
\[F_1 = XY^{d-1},\quad G_1 = aX^d+bY^d,\]
we have
\[\widetilde F_0 (aX^d + bY^d,Y^d) = Y^d = \frac{YF_1(X,Y)}{X}\]
and
\[\widetilde G_0 (aX^d + bY^d,Y^d) = aX^d + bY^d= G_1(X,Y).\]

\item \framebox{$n\Rightarrow n+1$}\enspace Assuming \eqref{eq:tildefgrel} for $n$, from \eqref{eq:tildefgrecurrence} we have
\begin{align*}
    \widetilde F_n (aX^d+bY^d,Y^d) &= \widetilde F_{n-1} (aX^d+bY^d,Y^d) \widetilde G_{n-1}(aX^d+bY^d,Y^d)^{d-1} \\
    &= \frac{YF_n(X,Y)}{X} G_n(X,Y)^{d-1} \\
    &= \frac{YF_{n+1}(X,Y)}{X}
\end{align*}
and
\begin{align*}
    \widetilde G_n(aX^d+bY^d,Y^d) &= aX^d \frac{\widetilde F_{n-1}(aX^d+bY^d,Y^d)^d}{Y^d} + b \widetilde G_{n-1} (aX^d+bY^d,Y^d)^d \\
    &= \frac{aX^d}{Y^d} \left(\frac{YF_n(X,Y)}{X}\right)^d + b G_n(X,Y)^d \\
    &= aF_n(X,Y)^d + bG_n(X,Y)^d = G_{n+1}(X,Y).
\end{align*}
\end{itemize}

Therefore, \eqref{eq:tildefgrel} holds for all $n \ge 1$. Now if $n>1$, it follows that
{\allowdisplaybreaks
\begin{align*}
    &\widetilde \Phi_{d,b,n}^* (aX^d + bY^d, Y^d ) \\*
    &= \prod_{m\mid n} \left( \widetilde F_{m-1} (aX^d+bY^d, Y^d) - \widetilde G_{m-1} (aX^d+ bY^d, Y^d) \right)^{\mu(n/m)} \\
    &= \prod_{m \mid n} \left(\frac{YF_m(X,Y)}{X} - G_m(X,Y)\right)^{\mu(n/m)} \\
    &= \prod_{m \mid n} \left( \frac{YF_m(X,Y)- XG_m(X,Y)}{X}\right)^{\mu(n/m)} \\*
    &= \Phi_{d,a,b,n}^* (X,Y).
\end{align*}
}
This completes the proof of Lemma~\ref{lemma:phibntilde}.
\end{proof}

By construction, the polynomial $\widetilde \Phi_{d,b,n}^*$ is homogeneous in $x,y$ with degree \[\nu_d(n):= \sum_{m\mid n} \mu\left(\frac{n}{m}\right) d^{m-1}.\]
Note that if $n>1$, then \[\nu_d(n) \equiv \sum_{m\mid n} \mu \left(\frac{n}{m}\right) = 0 \pmod {d-1},\] so \[\frac{\nu_d(n)}{d-1} \in \bZ.\]

Now we investigate the leading and final coefficients of $\widetilde \Phi_{d,b,n}^*$.

\begin{lemma}
\label{lemma:cnitopbottom}
Let $c_{n,i} \in \bZ[b]$ be the coefficients of $\widetilde \Phi_{d,b,n}^*$, i.e., \[\widetilde \Phi_{d,b,n}^* (x,y)= \sum_{i=0}^{\nu_d(n)} c_{n,i} x^i y^{\nu_d(n)-i}.\] Then for all $n\ge1$ we have \[c_{n,0}, c_{n,\nu_d(n)} \in \left\{\pm b^{\nu_d(n)/(d-1)}\right\}.\]
\end{lemma}

\begin{proof}
For $k\ge 0$, let \[\widetilde F_k - \widetilde G_k = \sum_{j=0}^{d^k} e_{k,j} x^j y^{d^k-j},\] where $e_{k,j} \in \bZ[b]$.
We prove by induction that $e_{k,0},e_{k,d^k} \in \bigl\{\pm b^{(d^k-1)/(d-1)}\bigr\}$ for all $k\ge 0$.

\begin{itemize}

\item \framebox{$k=0$}\enspace Since \[\widetilde F_0 - \widetilde G_0 = y-x,\] we have \[e_{0,0} = 1,\quad e_{0,1}=-1.\]

\item \framebox{$k=1$}\enspace Since \[\widetilde F_1 - \widetilde G_1 = by^d - xy^{d-1} + x^{d-1} y - bx^d,\] we have \[e_{1,0} = b,\quad e_{1,d}=-b.\]

\item \framebox{$k\Rightarrow k+1$, $k\ge 1$}\enspace It follows from \eqref{eq:tildefgrecurrence} that $xy$ divides both $F_k$ and $F_{k+1}$.
In particular,
\[\widetilde G_{k+1} = (x-by) \frac{\widetilde F_k^d}{y} + b \widetilde G_k^d \equiv b \widetilde G_k^d \pmod{xy},\] 
since $d\ge 2$.
Moreover,
\[e_{k,0} y^{d^k} + e_{k,d^k} x^{d^k} \equiv \widetilde F_k - \widetilde G_k \equiv -\widetilde G_k \pmod{xy},\]
and hence
\begin{align*}
    \widetilde F_{k+1} - \widetilde G_{k+1} &\equiv - \widetilde G_{k+1} \\
    &\equiv -b\widetilde G_k^d \\
    &\equiv -b \left(-e_{k,0} y^{d^k} - e_{k,d^k} x^{d^k}\right)^d \\
    & \equiv (-1)^{d+1} \left(be_{k,0}^d y^{d^{k+1}} + be_{k,d^k}^d x^{d^{k+1}}\right) \pmod{xy}.
\end{align*}
Therefore,
\[e_{k+1,0} = (-1)^{d+1} be_{k,0}^d \in \left\{\pm b^{(d^{k+1}-1)/(d-1)}\right\},\]
and similarly $e_{k+1,d^{k+1}} \in \bigl\{\pm b^{(d^{k+1}-1)/(d-1)}\bigr\}$.
\end{itemize}

It follows from \eqref{eq:newdyn} that
\[\sum_{m \mid n} \mu\left(\frac{n}{m}\right) \frac{d^{m-1}-1}{d-1} = \frac{1}{d-1} \sum_{m \mid n} \mu\left(\frac{n}{m}\right) (d^{m-1}-1) = \frac{\nu_d(n)}{d-1}, \]
which completes the proof of Lemma~\ref{lemma:cnitopbottom}.
\end{proof}




This leads us to prove a stronger result of \cite[Theorem 4.1]{manes}, where $b=\pm 1$.

\begin{theorem} \label{thm:psia}
Let $d\ge 2$ be an integer and $a \in \bQ^*$. Then the following holds.
\begin{enumerate}[label=(\alph*)]
    \item There is no fixed point of $\psi_{d,a,1}$ other than $z=0$.
    \item There are at most two (or three) fixed points of $\psi_{d,a,-1}$ including $z=0$ if $d$ is odd (or even, respectively).
    \item Any rational periodic cycle of $\psi_{d,a,1}$ has length $1$ or $2$. There is a rational periodic cycle of length $2$ only if $d$ is even, in which case there is at most one such cycle.
    \item There is no rational periodic cycle of $\psi_{d,a,-1}$ of length at least $2$.
\end{enumerate}
\end{theorem}

\begin{proof}
First, if $z$ is fixed by $\psi_{d,a,1}$, then \[\psi_{d,a,1} (z) = \frac{z}{az^d+1}=z,\] which is only possible if $z=0$. Therefore, $z=0$ is the only fixed point of $\psi_{d,a,1}$. On the other hand, \[\psi_{d,a,-1}(z)=\frac{z}{az^d-1},\] so the roots of $az^d-1=1$ are nonzero fixed points of $\psi_{d,a,-1}$. This proves (a) and (b).

Now suppose that $z_1,\ldots,z_n$ is a rational periodic cycle of $\psi_{d,a,1}$ of length $n \ge 2$.
Then $z_i\ne 0$ for each $i=1,\ldots,n$. Also, \[\Phi_{d,a,1,n}^*(z_i)=0\] by the definition of dynatomic polynomials, so \[\widetilde \Phi_{d,1,n}^*(az_i^d +1) = 0\] follows from Lemma \ref{lemma:phibntilde}.
Therefore, $az_i^d+1$ is a rational root of $\widetilde \Phi_{d,1,n}^* \in \bZ[z]$, whose leading coefficient and constant term have absolute value $1$ by Lemma \ref{lemma:cnitopbottom}.
Therefore, \[a z_i^d + 1 \in \{\pm 1\}\] for all $i$, which implies
\begin{equation} \label{eq:xiyi}
    z_i^d = -\frac{2}{a}
\end{equation}
since we assumed that $z_i\ne 0$.
If $d$ is odd, there is at most one $z_i\in \bQ$ satisfying \eqref{eq:xiyi}, which contradicts that $z_1,\ldots,z_n$ should be distinct.
Therefore, there is no rational periodic cycle of $\psi_{d,a,1}$ of length $n\ge 2$.
On the other hand, if $d$ is even, there are at most two $z_i \in \bQ$ satisfying \eqref{eq:xiyi}, which allows at most one rational periodic cycle of $\psi_{d,a,1}$ of length $2$. Note that for this cycle $z_1,z_2$, \eqref{eq:xiyi} implies \[az_1^d+1 = az_2^d + 1= -1,\] so $z_2 = -z_1$.

Similarly, if $z_1,\ldots,z_n$ is a rational periodic cycle of $\psi_{d,a,-1}$ of length $n\ge 2$, then \[az_i^d-1 \in \{\pm 1\}\] for all $i$, and consequently \[z_i^d = \frac{2}{a}.\] There are two roots, say $z_1,z_2$, if $d$ is even, but in this case \[az_1^d -1 = az_2^d -1 = 1,\] so $z_1=z_2$. Therefore, $\psi_{d,a,-1}$ cannot have any rational periodic cycle of length $2$ or longer.

\end{proof}

\begin{theorem}
    Let $d\ge 2$. Then for any $a\in \bQ^*$, \[\#\PrePer(\psi_{d,a,\pm 1}, \bP^1(\bQ))\le 6.\]

    \begin{proof}
        We first consider $\psi_{d,a,1}$. \cite{manes} covered the $d=2$ case, so we may assume that $d\ge 3$. Also, Theorem \ref{thm:psia} says that there are at most one periodic point $z=0$ if $d$ is odd, and at most three periodic point, $z=0$ and $z=\pm \alpha$ where $\alpha$ and $-\alpha$ form a periodic cycle of length $2$, if $d$ is even. It follows that \[a\alpha^d+2=0.\]

        Now observe that the preimage of $z=0$ is $\{0,\infty\}$. If $z\in \bQ$ satisfies $\psi_{d,a,1}(z)=\infty$, then $az^d+1=0$. This equation has at most one rational root $\beta$ if $d$ is odd, and at most two rational roots $\pm \beta$ if $d$ is even. Next, if $z\in \bQ$ satisfies \[\psi_{d,a,1}(z) =\frac{z}{az^d+1}=  \beta,\] then \[\frac{az^d}{(az^d+1)^d} = a \beta^d=-1.\] Therefore, $w=az^d+1$ is a rational root of \[w^d+w-1 = 0,\] which is impossible. The same argument works for $-\beta$ when $d$ is even, so there are at most four preperiodic points which are eventually mapped to $z=0$.

        Similarly, if $d$ is even and $z\in \bQ$ satisfies \[\psi_{d,a,1}(z)=\frac{z}{az^d+1} =\alpha,\] then \[\frac{az^d}{(az^d+1)^d} = a\alpha^d = -2.\] It follows that $w=az^d+1$ is a rational root of \[2w^d+w-1=0,\] which is impossible (here we use the assumption $d\ge 3$). Therefore, there is no rational strictly preperiodic points attached to the periodic cycle of length $2$. To sum up, there are at most six preperiodic points of $\psi_{d,a,1}$.

        For the other map $\psi_{d,a,-1}$, there are at most four preperiodic points in the backward orbit of $z=0$ by a similar argument. Now let $\gamma$ be a nonzero fixed point of $\psi_{d,a,-1}$, which satisfies $a\gamma^d-2=0.$ If $z\in \bQ$ satisfies \[\psi_{d,a,-1}(z)=\frac{z}{az^d-1}=\gamma\] then \[\frac{az^d}{(az^d-1)^d} = a\gamma^d = 2,\] so $w=az^d-1$ is a rational root of \[2w^d - w+1=0.\] This is impossible for any $d\ge 2$. Therefore, there are at most six preperiodic points as well.
    \end{proof}
\end{theorem}

Now we use Lemma \ref{lemma:cnitopbottom} to expand this result to $b$-values which are prime powers.

\begin{definition}
For a fixed prime $p$, let \begin{align*} \cP &:= \bigl\{ \sigma p^e : \sigma \in \{\pm 1\},\ e \in \bZ \bigr\}, \\ \cP_0^+ &:= \bigl\{ \sigma p^e : \sigma \in \{\pm 1\},\ e \in \bZ_0^+ \bigr\}, \\ \cP^+ &:= \bigl\{ \sigma p^e : \sigma \in \{\pm 1\},\ e \in \bZ^+ \bigr\}.\end{align*} We say that a polynomial $f(x)\in \bZ[x]$ is \emph{$\cP^+$-sided} if the leading coefficient and constant term of $f$ are in $\cP^+$.
\end{definition}

\begin{lemma} \label{lemma:pplemma}
Fix a prime $p$, $d \ge 3$, $a\in \bQ^*$ and $b\in \cP^+$. Suppose that $z_1 , z_2 \in \bQ$ satisfy $az_i^d+b \in \cP$ and $\psi_{d,a,b}(z_1) = z_2$. Then \[z_2 \in \{\pm z_1,\pm p^r z_1\}\] where \[r = \frac{\ord_p(b)}{d-1}.\] Moreover, $z_2\ne z_1$ is possible only if $d$ is even.

\begin{proof}
Let \[b=\sigma p^e \quad \text{and} \quad az_i^d+b = \sigma_i p^{k_i},\] where $\sigma,\sigma_i \in \{\pm 1\}$, $e\in \bZ^+$ and $k_i \in \bZ$. Then we have
\begin{align*}
    z_2 = \psi_{d,a,b}(z_1) = \frac{z_1}{az_1^d+b} \quad & \Longrightarrow \quad z_2 (az_1^d+b) = z_1 \\
    & \Longrightarrow \quad az_2^d (az_1^d+b)^d = az_1^d \\
    & \Longrightarrow \quad (\sigma_2 p^{k_2} - b) (\sigma_1 p^{k_1})^d = \sigma_1 p^{k_1} - b \\
    & \Longrightarrow \quad (\sigma_2 p^{k_2} - \sigma p^e ) (\sigma_1 p^{k_1})^d = \sigma_1 p^{k_1} - \sigma p^e \\
    & \Longrightarrow \quad (\sigma_1^d \sigma_2) p^{dk_1 + k_2 } - (\sigma \sigma_1^d) p^{dk_1 + e} - \sigma_1 p^{k_1} + \sigma p^e = 0.
\end{align*}
Therefore, $x=p$ is a root of a polynomial equation
\begin{equation} \label{eq:pequation}
f(x) := x^N\left((\sigma_1^d \sigma_2) x^{dk_1+k_2} - (\sigma \sigma_1^d) x^{dk_1+e} - \sigma_1 x^{k_1} + \sigma x^e \right) = 0
\end{equation}
for sufficiently large $N$ so that every exponent of $x$ is nonnegative.
We determine all possible combinations $(k_1,k_2,\sigma_1,\sigma_2)$ such that \eqref{eq:pequation} can have $x=p$ as a root.
If $p\ge 5$, \eqref{eq:pequation} cannot have $x=p$ as a root unless $f \equiv 0$ as a polynomial.
This can happen only if four terms in $f$ are cancelled in two pairs, which gives the following three cases. Note that $z_2 = \psi_{d,a,b} (z_1)$ implies \[\sigma_1 p^{k_1} z_2 = z_1 \quad \Longrightarrow \quad z_2 = \pm p^{-k_1} z_1.\]

\begin{enumerate}[label=\textbf{(\Alph*)}]
    \item If \[(\sigma_1^d \sigma_2) x^{dk_1+k_2} = (\sigma \sigma_1^d) x^{dk_1+e}\quad \text{and} \quad \sigma_1 x^{k_1} = \sigma x^e,\] then $k_1 =k_2 =e$ by comparing exponents and $\sigma_1 = \sigma_2 = \sigma$ by comparing coefficients.
    This implies that \[az_1^d + b = az_2^d + b = b\] so $z_1 = z_2 = 0$.
    
    \item If \[(\sigma_1^d \sigma_2) x^{dk_1+k_2} = \sigma_1 x^{k_1} \quad \text{and} \quad (\sigma \sigma_1^d) x^{dk_1+e} = \sigma x^e,\] then $k_1 = k_2 =0$ by comparing exponents and $\sigma_1 = \sigma_2$ by comparing coefficients. Then \[ az_1^d + b = az_2^d +b,\] so $z_2 = \pm z_1$ and $z_2 = -z_1$ is possible only if $d$ is even.
    
    \item If \[(\sigma_1^d \sigma_2) x^{dk_1+k_2} = -\sigma x^e \quad \text{and} \quad (\sigma \sigma_1^d) x^{dk_1+e}=-\sigma_1 x^{k_1},\] then $k_1 = -e/(d-1)=-r$ and $k_2 = (2d-1)e/(d-1)=(2d-1)r$ by comparing exponents and $\sigma = -\sigma_1^{d-1}$ and $\sigma_1 = \sigma_2$ by comparing coefficients. Then \[z_2 = \pm p^r z_1.\]
\end{enumerate}

If $p<5$, four terms in $f$ are cancelled in two pairs unless $(dk_1+k_2,dk_1+e,k_1,e)$ is a permutation of $(m+2,m+1,m,m)$ (when $p=2$) or $(m+1,m,m,m)$ (when $p=3$) for some $m\in \bZ$. However, since $d\ge 3$, it follows that $dk_1+e$ and $e$ should be equal, i.e., $k_1=0$, in any case. Therefore, $z_2 = \pm z_1$ as well.
\end{proof}
\end{lemma}

\begin{remark}
    For $d=2$ and $p=2$, it may be possible that $k_1=\pm 1$. However, $r=\ord_p(b) = e$ in this case, so if $k_1 = -1$ and $e=1$ then \[z_2 = \pm p z_1 \in \{\pm z_1,\pm p^r z_1\}.\] Excluding such case, the only exception for Lemma \ref{lemma:pplemma} happens when \[(p,e,k_1,k_2,\sigma_1,\sigma_2) = (2,1,1,0,-\sigma,\sigma),\] so \[b=2 \quad \text{and} \quad az_1^2=-4,\ az_2^2 = -1\] or \[b=-2 \quad \text{and} \quad az_1^2=4,\ az_2^2 = 1.\] In the first case, \[\psi_{2,a,2} (z_2) = \frac{z_2}{az_2^2+2} = z_2,\] so $z_2$ is a fixed point of $\psi_{2,a,2}$. In the second case, \[\psi_{2,a,-2} (z_2) = \frac{z_2}{az_2^2-2} = -z_2,\] and \[\psi_{2,a,-2}(-z_2) = \frac{-z_2}{a(-z_2)^2-2} = z_2.\] Therefore, $z_2$ is a periodic point of exact period $2$. Note that $z_1$ and $z_2$ cannot be elements of a periodic cycle of $\psi_{2,a,b}$ in any case.
\end{remark}

\begin{theorem} \label{thm:cyclelength}
Fix a prime $p$, $d\ge 2$, $a \in \bQ^*$, and $b\in \cP_0^+$.
\begin{enumerate}[label=(\alph*)]
    \item If $d$ is odd, then every rational periodic point of $\psi_{d,a,b}$ is fixed.
    \item If $d$ is even, then every rational periodic cycle of $\psi_{d,a,b}$ has length at most $2$, and any rational periodic cycle of length $2$ should be of the form \[z \longleftrightarrow -z.\]
\end{enumerate}
\end{theorem}
\begin{proof}
We already proved for $b=\pm 1$ in Theorem \ref{thm:psia}, so we may assume that $b\in \cP^+$.

Suppose that $z_1,\ldots,z_n$ is a rational periodic cycle of length $n\ge 2$.
Using a similar argument as in Theorem \ref{thm:psia}, $az_i^d+b$ is a rational root of $\widetilde \Phi_{d,b,n}^*$, whose leading coefficient and constant term are one of \[\pm b^{\nu_d(n)/(d-1)} = \pm p^{e\nu_d(n)/(d-1)}.\]
Therefore,
\begin{equation} \label{eq:azi}
    az_i ^d + b = \sigma_i p^{k_i}
\end{equation}
for some $\sigma_i \in \{\pm 1\}$ and $k_i\in \bZ$ satisfying
\begin{equation} \label{eq:krange}
    |k_i|\le \frac{e\nu_d(n)}{d-1}.
\end{equation}

Then from Lemma \ref{lemma:pplemma}, $d$ should be even and \[z_2 \in \{-z_1,\pm p^r z_1\}.\] (Note that, from the remark after Lemma \ref{lemma:pplemma}, the exceptional case when $d=2$ does not happen.) Now repeat the same process with $z_{j+1} = \psi_{d,a,b}(z_{j})$ for $j=2,\ldots,n$ (let $z_{n+1}=z_1$), then \[z_{j+1} \in \{-z_j, \pm p^r z_j.\}.\] Since $r>0$, the only possibility is that $n=2$ and $z_2=-z_1$.

\end{proof}

We can further determine all possible rational preperiodic portraits. If $z$ is a fixed point, then $z=0$ or $z=\alpha$ with \[a\alpha^d+b=1.\] If $d$ is odd, there are at most one nonzero fixed point $\alpha\in \bQ$; if $d$ is even, there are at most two nonzero fixed points $\pm \alpha \in \bQ$.

If $d$ is even, there may be rational periodic cycles of length $2$. However, in the proof of Theorem \ref{thm:cyclelength}, we proved that any such periodic cycle should consist of $z$ and $-z$. Consequently, there are at most one rational periodic cycle of length $2$, where $z=\beta$ with \[a\beta^d+b = -1.\] As a result, we have the following diagram for all possible periodic cycles. Note that the dashed arrow implies that it is possible only if $d$ is even.
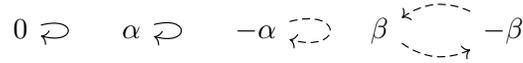
\begin{center}
    \begin{tikzcd}
        0 \arrow[loop right] & \alpha \arrow[loop right] & -\alpha \arrow[dashed,loop right] & \beta \arrow[dashed,bend right]{r} & -\beta \arrow[dashed,bend right]{l}
    \end{tikzcd}
\end{center}

It remains to investigate the strictly preperiodic points.

\begin{theorem}
Let $p$ be a prime, and $d\ge 3$. Then for any $a\in \bQ^*$ and $b\in \cP_0^+$, \[\#\PrePer(\psi_{d,a,b},\bP^1(\bQ))\le 6.\]

\begin{proof}
First, the preimage of $0$ is $\{0,\infty\}$, and $\psi_{d,a,b}(z)=\infty$ if and only if $az^d+b=0$. Therefore, there is at most one rational preimage of $\infty$, say $\gamma$, if $d$ is odd. If $d$ is even, there are at most two, which can be denoted by $\pm \gamma$. Therefore, we have the following diagram.
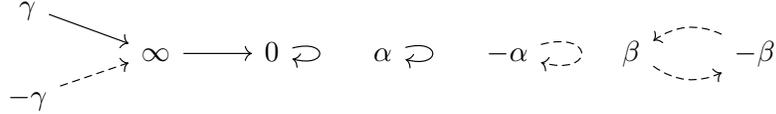
\begin{center}
    \begin{tikzcd}[row sep=0pt]
        \gamma \arrow{dr} \\
        &\infty \arrow{r} &0 \arrow[loop right] & \alpha \arrow[loop right] & -\alpha \arrow[dashed,loop right] & \beta \arrow[dashed,bend right]{r} & -\beta \arrow[dashed,bend right]{l} \\
        -\gamma \arrow[dashed]{ur}
    \end{tikzcd}
\end{center}

We now claim that there is no other rational preperiodic point. 

\begin{enumerate}[label=\textbf{(\Alph*)}]
    \item Let $z\in \bQ$ be satisfying $\psi_{d,a,b} (z)=\gamma$. Then \[a\left(\frac{z}{az^d+b}\right)^d + b = a\gamma^d+b=0,\] so \[az^d+b(az^d+b)^d = 0.\] Therefore, $w=az^d+b$ should be a rational root of \[bw^d+w-b=0.\] This is a $\cP^+$-sided polynomial, and its $p$-Newton polygon has two line segments with slope $-e$ and $e/(d-1)$, where $e=\ord_p(b)$, so $w=\pm p^e=\pm b$ or $w=\pm p^{-e/(d-1)}$. It is not hard to see that both are impossible, so there is no $z\in \bQ$ such that $\psi_{d,a,b}(z)=\gamma$. The same argument works for $-\gamma$ when $d$ is even.
    
    \item Let $z(\ne \pm \alpha)\in \bQ$ be satisfying $\psi_{d,a,b}(z)=\alpha$. Then \[a\left(\frac{z}{az^d+b}\right)^d + b = a\alpha^d+b=1,\] so \[az^d+(b-1)(az^d+b)^d=0.\] Therefore, $w=az^d+b$ should be a rational root of \[(b-1)w^d+w-b=0.\] $w\ne 1$ as $z\ne \pm \alpha$, so \[(b-1)w^{d-1} +(b-1)w^{d-2}+\cdots+(b-1)w+b=0.\] For any prime divisor $q$ of $b-1$, the $q$-Newton polygon of the previous polynomial consists of one line segment with slope $\ord_q(b-1)/(d-1)$, which implies \[\ord_q(w)=-\frac{\ord_q(b-1)}{d-1}\] for each $q$. Therefore, \[w = \pm \frac{p^k}{\sqrt[d-1]{|b-1|}}\] for some $0\le k\le e$. However, considering the $p$-Newton polygon of the same polynomial, it turns out that $k=0$ or $e$, and both are impossible. The same argument works for $-\alpha$ when $d$ is even.
    
    \item Let $z(\ne \pm \beta) \in \bQ$ be satisfying $\psi_{d,a,b}(z)=\beta$, when $d$ is even. Then \[a\left(\frac{z}{az^d+b}\right)^d + b = a\alpha^d+b=-1,\] so so \[az^d+(b+1)(az^d+b)^d=0.\] Therefore, $w=az^d+b$ should be a rational root of \[(b+1)w^d+w-b=0.\] $w\ne -1$ as $z\ne \pm \beta$, so \[(b+1)w^{d-1} -(b+1)w^{d-2}+\cdots+(b+1)w-b=0.\] Similarly as in the previous case, \[w = \pm \frac{p^k}{\sqrt[d-1]{|b+1|}}\] for $k=0$ or $e$, both of which are impossible.
\end{enumerate}

Finally, we point out that not all of $\alpha,\beta,\gamma$ can be rational when $d$ is even. This is because if $\alpha,\beta,\gamma \in \bQ$, then \[a\beta^d+b=1,\quad a\gamma^d+b=0,\quad a\alpha^d+b=1,\] so $(\beta^d,\gamma^d,\alpha^d)$ forms an arithmetic progression. By Merel-Darmon \cite{merel}, this is impossible when $d$ is even. Therefore, there exists at most $6$ rational preperiodic points of $\psi_{d,a,b}$.

\end{proof}
\end{theorem}

\bibliography{biblio.bib}

\begin{thebibliography}{MSW17}

\bibitem[FPS97]{fps}
E.~V. Flynn, Bjorn Poonen, and Edward~F. Schaefer.
\newblock Cycles of quadratic polynomials and rational points on a genus-$2$
  curve.
\newblock {\em Duke Mathematical Journal}, 90(3):435--463, 1997.

\bibitem[LMT14]{manes}
Alon Levy, Michelle Manes, and Bianca Thompson.
\newblock Uniform bounds for preperiodic points in families of twists.
\newblock {\em Proceedings of the American Mathematical Society},
  142(9):3075--3088, 2014.

\bibitem[Loo21]{looper}
Nicole~R. Looper.
\newblock Dynamical uniform boundedness and the $abc$-conjecture.
\newblock {\em Inventiones mathematicae}, 225(3):1--44, 2021.

\bibitem[Man08]{manes2}
Michelle Manes.
\newblock $\mathbb{Q}$-rational cycles for degree-2 rational maps having an
  automorphism.
\newblock {\em Proceedings of the London Mathematical Society}, 96(3):669--696,
  2008.

\bibitem[MD97]{merel}
Loic Merel and Henri Darmon.
\newblock Winding quotients and some variants of fermat's last theorem.
\newblock {\em Journal für die reine und angewandte Mathematik (Crelles
  Journal)}, 1997(490-491):81--100, 1997.

\bibitem[Mor98]{morton}
Patrick Morton.
\newblock Arithmetic properties of periodic points of quadratic maps, ii.
\newblock {\em Acta Arithmetica}, 87:89--102, 1998.

\bibitem[MS94]{ubc}
Patrick Morton and Joseph~H. Silverman.
\newblock Rational periodic points of rational functions.
\newblock {\em International Mathematics Research Notices}, 2:97--110, 1994.

\bibitem[MSW17]{auto}
Nikita Miasnikov, Brian Stout, and Phillip Williams.
\newblock Automorphism loci for the moduli space of rational maps.
\newblock {\em Acta Arithmetica}, 180(3):267--296, 2017.

\bibitem[Nor50]{northcott}
D.~G. Northcott.
\newblock Periodic points on an algebraic variety.
\newblock {\em Annals of Mathematics}, 51(1):167--177, 1950.

\bibitem[Poo98]{poonen}
Bjorn Poonen.
\newblock The classification of rational preperiodic points of quadratic
  polynomials over $\mathbb{Q}$: a refined conjecture.
\newblock {\em Mathematische Zeitschrift}, 228(1):11--29, 1998.

\bibitem[Sto08]{stoll}
Michael Stoll.
\newblock Rational 6-cycles under iteration of quadratic polynomials.
\newblock {\em LMS Journal of Computation and Mathematics}, 11:367--380, 2008.

\end{thebibliography}
\bibliographystyle{alpha}










\end{document}